\documentclass[12pt, a4]{amsart}
\usepackage{amscd, amsmath, amssymb}
\usepackage{fullpage}
\usepackage{mathrsfs,epsfig}

\theoremstyle{plain}
\newtheorem{thm}{Theorem}[section]

\theoremstyle{definition}
\newtheorem{defn}[thm]{Definition}
\newtheorem{ex}[thm]{Example}

\numberwithin{equation}{section}

\newcommand{\R}{\mathbb{R}}

 \usepackage{color}

\begin{document}

\title[Elliptic equations with functional dependence]{On the solvability of parameter-dependent elliptic functional BVPs on annular-like domains} 

\date{}

\author[A. Calamai]{Alessandro Calamai}
\address{Alessandro Calamai, 
Dipartimento di Ingegneria Civile, Edile e Architettura,
Universit\`{a} Politecnica delle Marche
Via Brecce Bianche
I-60131 Ancona, Italy}%
\email{a.calamai@univpm.it}%

\author[G. Infante]{Gennaro Infante}
\address{Gennaro Infante, Dipartimento di Matematica e Informatica, Universit\`{a} della
Calabria, 87036 Arcavacata di Rende, Cosenza, Italy}%
\email{gennaro.infante@unical.it}%

\begin{abstract} 
We investigate the existence of nontrivial solutions of parameter-dependent elliptic equations with deviated argument in annular-like domains in
$\mathbb{R}^{n}$, with $n\geq 2$, subject to functional boundary conditions.
In particular we consider a boundary value problem that may be used to model heat-flow problems.
We obtain an existence result by means of topological methods; in particular,
we make use of a recent variant in affine cones of the celebrated Birkhoff--Kellogg theorem. Using an ODE argument, we illustrate in an example the applicability of our theoretical result. 
\end{abstract}



\maketitle

\section{Introduction}

Our purpose is to study functional boundary value problems (BVPs) associated with elliptic equations in suitable
``annular-like'' domains; in particular we consider the parametric, functional BVP 
\begin{equation}
 \label{bvp-intro}
 \left\{
\begin{array}{ll}
L u=\lambda f(x, u, u_\sigma), & \ \text{in}\ \Omega, \\
u(x)=\psi (x), & \ \text{in}\ \overline \Omega_1, \\
u(x)=\lambda \zeta (x) B[u], & \ \text{on}\ \Gamma_2,
\end{array}
\right.
\end{equation}
where  $\Omega \subset \mathbb{R}^{n}$ with $n\geq 2$,
 $\Omega = \Omega_2 \setminus \overline\Omega_1$ with
$\Omega_1$, $\Omega_2$ bounded domains with sufficiently regular boundary, $\overline \Omega_1 \subset \Omega_2$,  $\partial \Omega_1 = \Gamma_1$, 
$\partial \Omega_2 = \Gamma_2$ and
$\Gamma_1$,  $\Gamma_2$ suitable manifolds. Moreover
 $\lambda$ is a real parameter, $L$ is a strongly uniformly elliptic operator, 
$\psi$ and $\zeta$ are continuous, $B$ is a suitable functional (see Section \ref{sec-set} for details).
Concerning the right-hand side appearing in the differential equation
\begin{equation}
 \label{eq-intro}
L u=\lambda f(x, u, u_\sigma),  \  x\in \Omega,
\end{equation}
by $f$ we mean a real-valued continuous function defined on $\overline \Omega\times\R\times\R$,
where $u_\sigma$ incorporates a deviated argument that can take into account the ``global'' behaviour of $u$, including the ``hole''  $\Omega_1$;
in particular, this setting covers the interesting case of spatial delays.

A motivation for studying this kind of BVPs is that they may occur in physical applications.
We illustrate this fact with the annulus $\Omega=\left\{x \in \mathbb{R}^{2}:1 < \| x \|_2 < 2\right\}$,
where $ \| \cdot \|_2$ is the Euclidean norm and consider the BVP
\begin{equation}
 \label{bvp-illustr}
 \left\{
\begin{array}{ll}
-\Delta u=\lambda f(x, u, u_\sigma), & \text{in}\ \Omega, \\
u(x)=\psi (x), & \text{if}\ \| x \|_2 \leq 1, \\
u(x)=\lambda u(\eta), & \text{if}\ \| x \|_2 =2,
\end{array}
\right.
\end{equation}
where $\eta \in \Omega$ is a given point in the interior of the annulus; this is illustrated in Figure \ref{fig-illustr}.
\begin{figure}[ht]
\centerline{\fbox{\epsfig{file=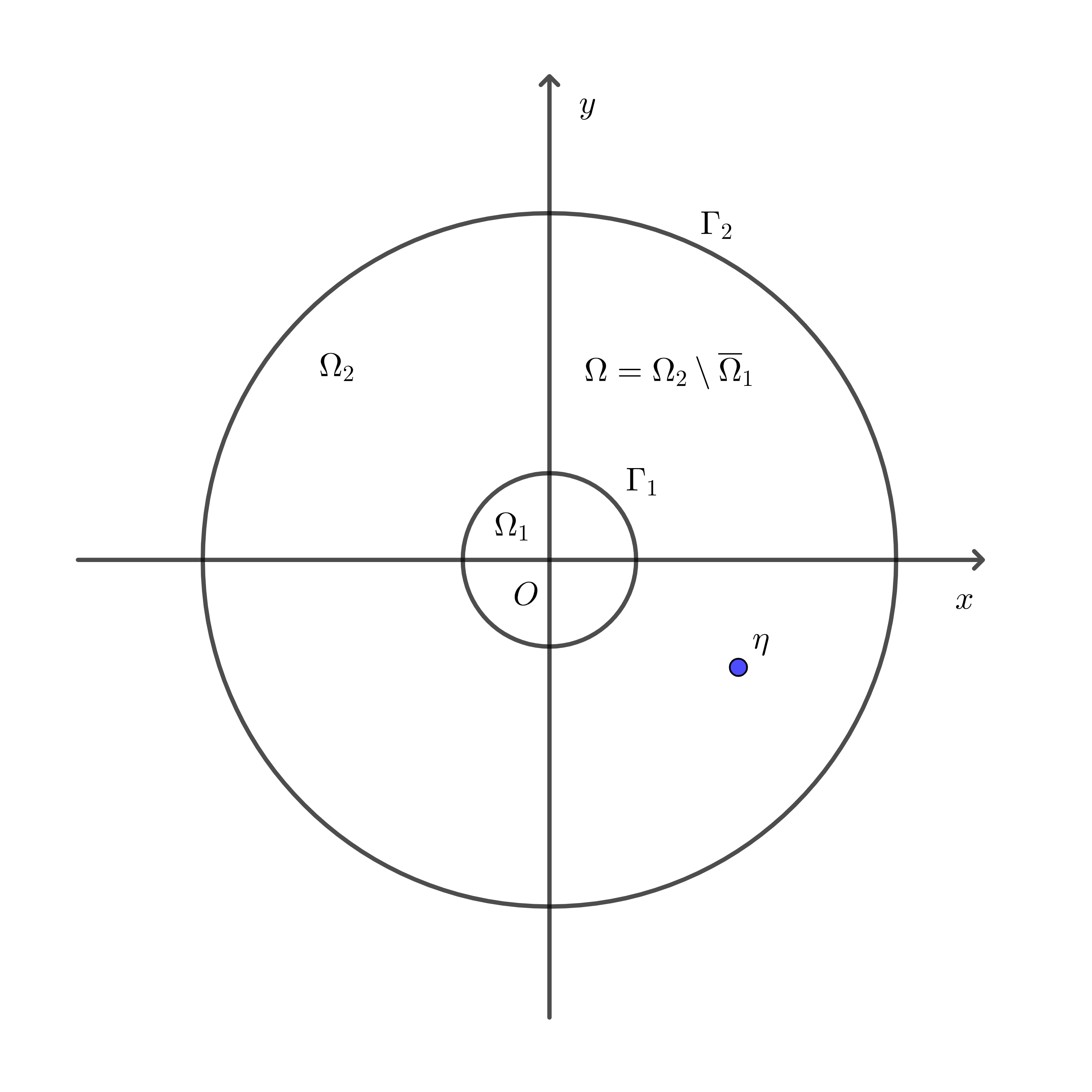, width = 6 cm} }}
\caption{}
\label{fig-illustr}
\end{figure}

The BVP \eqref{bvp-illustr} can be used as a model for the steady-states of the temperature of a heated annular plate.
In this model, the heat source can depend also on the ``global'' behaviour of the system, namely the heat source in a point $x$ of the annulus can also depend on the temperature measured at a point $\sigma(x)$ of the whole domain 
$\overline{\Omega}_2$ (and in particular it can take into account the values of the datum $\psi$).
Moreover, a controller located in the exterior border of the annulus adds or removes heat in manner proportional to the temperature registered by a sensor located at the point $\eta$ of the annulus.
For one-dimensional similar heat-flow problems see for example \cite{webb-therm} and references therein, for deviated arguments see \cite{ac-gi-at-tmna, acgi4, rub-rod-lms}, for the elliptic case see~\cite{sBaCgI, gi-tmna}.

Elliptic equations with non-local terms represent a widely investigated subject, also in view of applications.
For example, we mention equations of Kirchhoff type, see for example the review by Ma~\cite{ToFu} as well as 
the recent paper~\cite{gi-jepeq} and references therein. Concerning PDEs with deviated (or more general \emph{functional}) arguments, we refer for example to~\cite{ChenMa, gi-jepeq, Ros2019, Sim2018, Sk-book, Sk2016}.
On the other hand, BVPs associated with ODEs with deviated arguments have been studied by many authors with different techniques: we refer for example to~\cite{ac-gi-at-tmna} and references therein.

Parametric BVPs similar to \eqref{bvp-intro} have also been studied in association to elliptic functional differential
equations subject to functional BCs (see for example~\cite{sBaCgI,gi-jepeq}).
Note that in a local framework 
--  that corresponds to the choice $\sigma(x)=x$ in \eqref{eq-intro} --
and in the case of the Laplacian
under Dirichlet boundary conditions,
one gets the classical Gelfand-type problem: see 
the Introduction of \cite{Bonanno} for a recent review on this topic. In the context of PDEs, there is interest to the problem of parameter-dependent BCs, see for example a ``nonlinear spectral problem'' considered in Section 4.7 of the classical book by Pao~\cite{Pao}.

We wish to point out that that the functional boundary operator $B$ in \eqref{bvp-intro}
is fairly general and can be used to deal with nonlinear and nonlocal BCs; the latter topic is quite fashionable both in the context of ODEs and of PDEs, for example see 
 the reviews~\cite{Cabada1, Conti, rma, sotiris, Stik, Whyburn} and the papers~\cite{Goodrich3, Goodrich4, kttmna, ktejde, Pao-Wang, Picone, jw-gi-jlms}.

Regarding the nature of the domain, that is the annular-like region, for the elliptic context we refer the reader to the classical work of Amann and  L\'opez-G\'omez~\cite{Amann-LG-JDE98} and the more recent paper of Kowalczyk and co-authors~\cite{KPV}.

Our main tool is Theorem~\ref{BK-transl-norm}, which is a version in affine cones of the celebrated Birkhoff--Kellogg theorem. The proof of Theorem~\ref{BK-transl-norm} is obtained by topological methods  in the recent paper \cite{acgi2}.
Applications of Theorem~\ref{BK-transl-norm} in the context of ODEs have been considered in~\cite{acgi3, acgi4, CIR24}.
We remark that the setting of affine cones seems to be helpful when dealing with differential equations with delay effects; 
as far as we are aware, this is the first time that this approach is ultilized in the context of PDEs.
Here we obtain the existence of nontrivial solutions $(u,\lambda)$ of the BVP~\eqref{bvp-intro}
such that $u$ belongs to
a suitable traslate of a cone  of nonnegative functions.

We close the paper with an example, set in $\mathbb{R}^{2}$, where the nonlinearity $f$ that occurs is not radially symmetric; nevertheless we exploit the symmetry of the domain and use an ODE technique in order to construct suitable lower bounds that are a key ingredient to apply our theoretical result.

\section{Some known facts and setting of the problem} \label{sec-set}

Let $\Omega \subset \mathbb{R}^{n}, (n\geq 2)$ be an annular-like domain,
that is, $\Omega = \Omega_2 \setminus \overline\Omega_1$,
where $\Omega_1$, $\Omega_2$ are bounded domains with smooth boundary such that $\overline \Omega_1$
is strictly contained into $\Omega_2$.
Denote by $\Gamma_1 = \partial \Omega_1$ and by
$\Gamma_2 = \partial \Omega_2$; then, 
$\Gamma_1$ and $\Gamma_2$ are smooth compact manifolds without boundary (cf. \cite[Sect. 4]{Amann-rev}, see also \cite[Chap. 2]{PW}).
Given $\alpha\in (0,1)$ we denote by $C^{\alpha}(\overline{\Omega})$ the space of all $\alpha$-H\"{o}lder continuous functions $g:\overline{\Omega}\to \mathbb{R}$ and, for every $k\in \mathbb{N}$, we denote by $C^{k+\alpha}(\overline{\Omega})$ the space of all functions $g\in C^{k}(\overline{\Omega})$ such that all the partial derivatives of $g$ of order $k$ are $\alpha$-H\"{o}lder continuous in $\overline{\Omega}$ (for more details see \cite[Examples~1.13 and~1.14]{Amann-rev}).

We study the parametric, functional BVP
\begin{equation}
 \label{bvp}
 \left\{
\begin{array}{lc}
L u=\lambda f(x, u, u_\sigma), & \text{in}\ \Omega, \\
u(x)=\psi (x), & \text{on}\ \overline \Omega_1, \\
u(x)=\lambda \zeta (x) B[u], & \text{on}\ \Gamma_2,
\end{array}
\right.
\end{equation}
where $L$ is a strongly uniformly elliptic operator, namely
\begin{equation*}
L u(x)=-\sum_{i,j=1}^n a_{ij}(x)\frac{\partial^2 u}{\partial x_i \partial x_j}(x)+\sum_{i=1}^n a_{i}(x) \frac{\partial u}{\partial x_i} (x)+a(x)u(x), \quad \mbox{for $x\in \Omega$}
\end{equation*}
with coefficients $a_{ij},a_{i},a\in C^{\alpha}(\overline{\Omega})$ for $i,j=1,2,\ldots,m$, $a(x)\ge 0$ on $\overline{\Omega}$, $a_{ij}(x)=a_{ji}(x)$ on $\overline{\Omega}$ for $i,j=1,2,\ldots,m$; moreover there exists $\mu_{0}>0$ such that 
$$\sum_{i,j=1}^n a_{ij}(x)\xi_i \xi_j\ge \mu_{0} \|\xi\|_n^2 \quad \mbox{for $x\in \Omega$ and $\xi=(\xi_1,\xi_2,\ldots,\xi_n)\in\R^n$},$$
where $ \| \cdot \|_n$ is the Euclidean norm.

Concerning the other terms appearing in~\eqref{bvp}, $\lambda>0$ is a real parameter,
$f$ is a real-valued continuous function defined on $\overline \Omega\times\R \times\R$,
$\psi:  \overline \Omega_1 \to \R$ is continuous,
$\zeta:  \Gamma_2 \to \R_+$ is continuous, moreover
$B$ is a suitable functional defined on the space $C(\overline{\Omega}_2)$,
endowed with the standard supremum norm  $ \| \cdot \|_{\infty} $.

Finally, given a continuous map
$\sigma: \overline{\Omega} \to \overline{\Omega}_2$ and $u\in C(\overline{\Omega}_2)$,
 with the notation $u_\sigma$ we mean the composition
\[
u_\sigma: \overline{\Omega} \to \R \quad u_\sigma(x)=u(\sigma(x));
\]
hence $u_\sigma \in C(\overline{\Omega})$ is a map which may take into account the global behaviour of $u$.

We apply the classical elliptic theory to the following auxiliary Dirichlet problems.
In fact, it is known that, under the previous conditions, a strong maximum principle holds
(see~\cite{Amann-rev}, Section 4 of Chapter 1;  see also Chapter 2 of~\cite{PW}).
Furthermore, given $\mathfrak f\in C(\overline{\Omega})$, the homogeneous Dirichlet problem on $\Omega$
 \begin{equation}
 \label{eqelliptic.bc.hom}
 \left\{
\begin{array}{ll}
  Lu(x)=\mathfrak f(x), & x\in \Omega, \\
  u(x)=0, & x\in  \Gamma_1 \cup \Gamma_2,
\end{array}
\right.
\end{equation}
admits a unique classical solution $u_{\mathfrak f}\in C^{1,\alpha}(\overline{\Omega})$.
Moreover, the BVP
 \begin{equation}
 \label{eqelliptic.bc.in}
 \left\{
\begin{array}{ll}
  Lu(x)=0, & x\in \Omega, \\
  u(x)=\psi(x), & x\in \Gamma_1, \\
u(x)=0, & x\in \Gamma_2,
\end{array}
\right.
\end{equation}
admits a unique solution $\delta \in C^{1,\alpha}(\overline{\Omega})$, and the BVP
 \begin{equation}
 \label{eqelliptic.bc.out}
 \left\{
\begin{array}{ll}
  Lu(x)=0, & x\in \Omega, \\
  u(x)=0, & x\in \Gamma_1, \\
u(x)=\zeta(x), & x\in \Gamma_2,
\end{array}
\right.
\end{equation}
also admits a unique solution $\gamma \in C^{1,\alpha}(\overline{\Omega})$ such that $\gamma(x)\geq 0$ for every $x\in\Omega$.

The solution operator associated to the BVP~\eqref{eqelliptic.bc.hom}, $\mathcal G :C(\overline{\Omega})\to C^{1}(\overline{\Omega})$, defined as $\mathcal G (\mathfrak f)=u_{\mathfrak f}$ is linear and continuous.
Moreover, we observe that we can uniquely extend $\mathcal G$ to another operator,
denoted again by the same name,
$\mathcal G :C(\overline{\Omega})\to C(\overline{\Omega}_2)$, by considering a trivial continuous extension of a
function $v \in C(\overline{\Omega})$ such that $v|_{\Gamma_1}\equiv 0$, that is
 \begin{equation*}
\tilde v(x) = \left\{
\begin{array}{ll}
  \tilde v(x)= v(x), & x\in \overline \Omega, \\
  \tilde v(x)=0, & x\in  \Omega_1.
\end{array}
\right.
\end{equation*}
For example notice that $\tilde \gamma \in C(\overline{\Omega}_2)$, where $\gamma$ is the unique solution of the BVP \eqref{eqelliptic.bc.out}.
The operator $\mathcal G$ is continuous, linear and compact operator and leaves the cone of positive functions invariant (see, for example~\cite{Amann-rev}, Section 4 of Chapter 1).

Some further notation is in order. 
We define
\begin{equation} \label{def-psihat}
\phi(x) :=
\begin{cases}
\psi(x),\ & x\in \overline \Omega_1, \\
\delta(x),\ & x\in \overline{\Omega}= \overline \Omega_2 \setminus  \Omega_1,
\end{cases}
\end{equation}
where $\delta$ is the unique solution of the BVP \eqref{eqelliptic.bc.in}. Observe that, by construction, the function $\phi \in C(\overline{\Omega}_2)$.  The function $\phi$ will play a key role in the sequel, since it will be the vertex of the affine cone that we use when applying Theorem~\ref{BK-transl-norm}.

We denote by $\mathcal{F}$
 the \emph{superposition (Nemytskii)} operator
 associated with maps $f$ and $\sigma$ as above; that is,
\[
\mathcal{F}:C(\overline{\Omega}_2)\to C(\overline{\Omega}),\quad
\mathcal{F}(u)(x) := f(x,u(x),u(\sigma(x))),\ x\in\overline{\Omega}.
\]
Note, in particular, that the operator $\mathcal G \circ \mathcal F$ is well-defined.

The above discussion justifies the following definition of \emph{solution} of the problem \eqref{bvp}.
\begin{defn} \label{def-sol-bvp}
We say that a function $u\in C(\overline{\Omega}_2)$ is a {\it (weak) solution}
of the problem \eqref{bvp}~if 
\[
u=\phi  + \lambda \Bigl( \mathcal G \circ \mathcal F (u)  +  \tilde \gamma B[u] \Bigr).
\]
\end{defn}

In other words, $u$ is a solution
of the problem \eqref{bvp} if and only if it is a fixed point of the operator $\phi  +  \lambda \mathcal{T}$, where 
\[
\mathcal{T}:C(\overline{\Omega}_2)\to C(\overline{\Omega}_2),\qquad
\mathcal{T}(u) := \mathcal G \circ \mathcal F (u)  +  \tilde \gamma B[u].
\]

Let now $(X,\| \, \|)$ be a real Banach space. A \emph{cone} $K$ of $X$  is a closed set with $K+K\subset K$, $\mu K\subset K$ for all $\mu\ge 0$ and $K\cap(-K)=\{0\}$.
For $y\in X$, the \emph{translate} of the cone $K$ is defined as
$$
K_y:=y+K=\{y+x: x\in K\}.
$$
Given a bounded and open (in the relative
topology) subset $\mathcal O$  of $K_y$, we denote by $\overline{\mathcal O}$ and $\partial \mathcal O$
the closure and the boundary of $\mathcal O$ relative to $K_y$.
Given an open
bounded subset $D$ of $X$ we denote $D_{K_y}=D \cap K_y$, an open subset of
$K_y$.

With these ingredients we can now state a Birkhoff–Kellogg type result in affine cones.

\begin{thm}[\cite{acgi2}, Corollary 2.4] \label{BK-transl-norm}
Let $(X,\| \, \|)$ be a real Banach space, $K\subset X$ be a cone and
 $D\subset X$ be an open bounded set with $y \in D_{K_y}$ and
$\overline{D}_{K_y}\ne K_y$. Assume that $\mathcal{F}:\overline{D}_{K_y}\to K$ is
a compact map and assume that 
$$
\inf_{x\in \partial D_{K_y}}\|\mathcal{F} (x)\|>0.
$$
Then there exist $x^*\in \partial D_{K_y}$ and $\lambda^*\in (0,+\infty)$ such that $x^*= y+\lambda^* \mathcal{F} (x^*)$.
\end{thm}

The proof of Theorem~\ref{BK-transl-norm} is derived from the
classical fixed point index theory for compact maps;
this is done in the context of affine cones, see also \cite{DM2014}.
 We refer a reader interested in the fixed point index to the review of Amann \cite{Amann-rev} and to the book by  Guo and Lakshmikantham \cite{guolak}.

By $K_0$ we denote the following cone of non-negative functions
in $C(\overline{\Omega}_2)$:
$$
K_0=\{u\in C(\overline{\Omega}_2): u(x)\geq 0\ \text{for every}\ x\in\overline{\Omega}_2\ \text{and}\ u(x)= 0 \ \text{for every}\ x\in \overline \Omega_1\}.
$$
Note that $K_0 \neq \{0\}$ since the map $\tilde \gamma$ introduced above is in $K_0$.

Given $\omega \in  C(\overline{\Omega}_2)$, let $K_\omega$ be the following translate of the cone $K_0$,
$$K_\omega=\omega + K_0 = \{\omega +u : u \in K_0\}.$$
\begin{defn}
Given $\omega \in  C(\overline{\Omega}_2)$ and $\rho>0$, we define the following subsets of $C(\overline{\Omega}_2)$:
$$K_{0,\rho}:=\{u\in K_0: \|u\|_{\infty} <\rho\},\quad 
K_{\omega,\rho}:= \omega + K_{0,\rho}.$$
\end{defn}

\section{Main result}

The following theorem provides an existence result for the problem~\eqref{bvp}: here we obtain a non-trivial solution within the cone $K_{\phi}$, where $\phi$ is as in \eqref{def-psihat}.

\begin{thm}\label{eigen}Let $\rho \in (0,+\infty)$ and assume the following conditions hold. 

\begin{itemize}

\item[$(a)$] 
There exist $\underline{\ell}_{\rho} \in C(\overline{\Omega})$ such that
\begin{equation*}
f(x,u,v)\ge \underline{\ell}_{\rho}(x) \ge 0,\ \text{for every}\ (x,u,v)\in \overline \Omega\times\R_+\times\R\
\text{with}\ \max\{u,|v|\}\leq 
 \rho+\|\phi\|_{\infty}.
\end{equation*}

\item[$(b)$] 
$B: \overline K_{\phi,\rho} \to \R_{+}$ is continuous and bounded, in particular let $\underline{b}_{\rho}\in [0,+\infty)$ be such that 
\begin{equation*}
B[u]\geq \underline{b}_{\rho},\ \text{for every}\ u\in \partial K_{\phi,\rho}.
\end{equation*}

\item[$(c)$] 
There exists $d_{\rho} \in (0,+\infty)$ such that
\begin{equation}\label{c-cond}
\sup_{x\in \overline{\Omega}} | \mathcal G(\underline{\ell}_{\rho})(x) + \underline{b}_{\rho} \gamma (x)|
\geq d_{\rho}.
\end{equation}
\end{itemize}

Then there exist $\lambda_\rho\in (0,+\infty)$ and $u_{\rho}\in \partial K_{\phi,\rho}$ that satisfy the problem \eqref{bvp}.
\end{thm}

\begin{proof}
First notice that,
due to the assumptions above, the operator $\mathcal T$ maps $\overline{K}_{\phi,\rho}$ into $K_0$ and is compact. 
In fact, since $\mathcal F$ is continuous and bounded and $\mathcal G$ is linear and compact, it follows that 
the operator $\mathcal G \circ \mathcal F$ is compact as well;
moreover, given $u \in \overline{K}_{\phi,\rho}$, we have $\mathcal G \circ \mathcal F (u)\in K_0$ by \emph{(a)} and the maximum principle.
On the other hand, assumption \emph{(b)} implies that $B$ is a continuous, finite rank operator with nonnegative values and, as we already stressed, $\tilde \gamma\in K_0$.

Take $u\in \partial K_{\phi,\rho}$, then for every $x\in \overline{\Omega}$
we have
\begin{equation*} %\label{lwest}
\mathcal G \circ \mathcal F(u)(x)+ \tilde \gamma(x) B[u] 
\ge  \mathcal G(\underline{\ell}_{\rho})(x) + \underline{b}_{\rho}  \gamma (x) .
\end{equation*}
Thus by \emph{(c)} we obtain
\[
\|\mathcal G \circ \mathcal F (u)  +  \tilde \gamma B[u]\|_{\infty}\geq
\sup_{x\in \overline{\Omega}} |\mathcal G \circ \mathcal F(u)(x)+ \tilde \gamma(x) B[u] |
\ge \sup_{x\in \overline{\Omega}} | \mathcal G(\underline{\ell}_{\rho})(x) + \underline{b}_{\rho} \gamma (x)|
\geq d_{\rho}.
\]
In particular, note that the RHS of the latter inequality does not depend on the particular $u$ chosen.
Therefore we have
$$
\inf_{u\in \partial K_{\phi,\rho}}\| \mathcal G \circ \mathcal F (u)  +  \tilde \gamma B[u]\|_{\infty}\geq d_{\rho}>0,
$$
and the result follows by Theorem~\ref{BK-transl-norm}.
\end{proof}

We illustrate now in an example the applicability of Theorem \ref{eigen}.

\begin{ex} \label{example1}
 Let us consider the following BVP in $\Omega \subset \mathbb{R}^{2}$,
 \begin{equation} \label{eq.example1}
  \begin{cases}
   -\Delta u(x_1,x_2) = \lambda (1+x_1^2) \, e^{-u(x_1,x_2)-u(\frac{x_1}2,\frac{x_2}2)}, & (x_1,x_2)\in \Omega, \\
   u(x_1,x_2) = x_1^2+x_2^2, & (x_1,x_2) \in \overline \Omega_1, \\
   u(x_1,x_2) = \lambda \int_{\Omega_2} u^2\,d x, & (x_1,x_2) \in \Gamma_2, \\
  \end{cases}
 \end{equation}
 where
 $$\Omega= \left\{(x_1,x_2): 1\leq\sqrt{x_1^2+x_2^2}\leq e\right\}. $$ 
 Note that in this case the function that occurs in the PDE is not radially symmetric. Nevertheless, in our case, it is convenient to take as a lower bound the radially symmetric function
 $$\underline{\ell}_{\rho}(x)=e^{-2(\rho+1)},$$
 and solve the ‘‘torsion problem''
  \begin{equation} \label{eq.example1r}
   \begin{cases}
   -\Delta u(x_1,x_2) = 1, & (x_1,x_2)\in \Omega, \\
   u(x_1,x_2) = 0, & (x_1,x_2) \in \Gamma_1 \cup \Gamma_2. \\
  \end{cases}
 \end{equation}
The solution of the BVP~\eqref{eq.example1r} is obtained by transforming the BVP~\eqref{eq.example1r} into a corresponding ODE problem.
Indeed, let $u(x_1,x_2):=\varphi(r)$ be a radial solution of~\eqref{eq.example1r},
with $r:=\sqrt{x_1^2+x_2^2}$.
Then, $\varphi$ solves
  \begin{equation*} %\label{ode.example1r}
   \begin{cases}
   -\varphi''(r)-\dfrac1r \varphi'(r) = 1, & r\in (1,e), \\
   \varphi(1) = 0, \ \varphi(e) = 0, \\
  \end{cases}
 \end{equation*}
which yields
 \begin{equation*}
 u_0(x_1,x_2)=\frac{1}{8}[\left({ e}^{2}-1\right) \ln\! \left(x_1^2+x_2^2\right)+2(1- x_1^2-x_2^2)].
 \end{equation*}
 The same ODE argument can be used to solve the BVP
 \begin{equation*}% \label{eq.example1rb}
  \begin{cases}
   -\Delta u(x_1,x_2) = 0, & (x_1,x_2)\in \Omega, \\
   u(x_1,x_2) = 0, & (x_1,x_2) \in \Gamma_1, \\
   u(x_1,x_2) = 1, & (x_1,x_2) \in \Gamma_2, \\
  \end{cases}
 \end{equation*}
which yields
 \begin{equation*} 
\gamma(x_1,x_2)=\frac{\ln\! \left(x_1^2+x_2^2\right)}{2}
 \end{equation*}
and the BVP
 \begin{equation*}% \label{eq.example1rb}
  \begin{cases}
   -\Delta u(x_1,x_2) = 0, & (x_1,x_2)\in \Omega, \\
   u(x_1,x_2) = 1, & (x_1,x_2) \in \Gamma_1, \\
   u(x_1,x_2) = 0, & (x_1,x_2) \in \Gamma_2, \\
  \end{cases}
 \end{equation*}
which gives
 \begin{equation*} 
\delta(x_1,x_2)=1-\frac{\ln\! \left(x_1^2+x_2^2\right)}{2}
 \end{equation*}
so that
 \begin{equation*} 
\phi(x_1,x_2)=
 \begin{cases}
    x_1^2+x_2^2, & (x_1,x_2) \in \overline \Omega_1, \\
   1-\dfrac{\ln\! \left(x_1^2+x_2^2\right)}{2}, & (x_1,x_2) \in \overline \Omega_2 \setminus  \overline \Omega_1. \\
  \end{cases}
 \end{equation*}

A direct  computation yields
 \begin{equation}\label{sup}
\sup_{(x_1,x_2)\in \Omega} |u_0(x_1,x_2)|=\frac{1}{8}\left[\left({ e}^{2}-1\right) \ln\! \left(\frac{{ e}^{2}-1}{2}\right)+3- e^{2}\right].
 \end{equation}
 From \eqref{sup} and the choice of  $\underline{b}_{\rho}=0$ gives
\begin{align*}
&\sup_{(x_1,x_2)\in \Omega} | \mathcal G(\underline{\ell}_{\rho})(x_1,x_2) + \underline{b}_{\rho} \gamma(x_1,x_2)|
\geq e^{-2(\rho+1)}\, \cdot \,\sup_{(x_1,x_2)\in \Omega} |u_0(x_1,x_2)| =\\
&=e^{-2(\rho+1)}\frac{1}{8}\left[\left({ e}^{2}-1\right) \ln\! \left(\frac{{ e}^{2}-1}{2}\right)+3- e^{2}\right]=:
d_{\rho}>0,
\end{align*}
which implies that \eqref{c-cond} is satisfied for every $\rho \in (0,+\infty)$.

Thus we can apply Theorem~\ref{eigen}, obtaining uncountably many pairs of solutions and parameters $(u_{\rho}, \lambda_{\rho})$ for the BVP~\eqref{eq.example1}.

\end{ex}

\section*{Acknowledgements}
The authors are members of the ``Gruppo Nazionale per l'Analisi Matematica, la Probabilit\`a e le loro Applicazioni'' (GNAMPA) of the Istituto Nazionale di Alta Matematica (INdAM).
G.~Infante is a member of the UMI Group TAA  ``Approximation Theory and Applications''.  
The authors were partly funded by the Research project of MUR - Prin 2022 “Nonlinear differential problems with applications to real phenomena” (Grant Number: 2022ZXZTN2).

\end{document}